\documentclass[12pt]{amsart}
\usepackage[margin=1.25in]{geometry}
\usepackage{graphicx,cancel}
\usepackage{graphics}
\usepackage[titletoc]{appendix}
\usepackage{color, soul}
\usepackage{amsmath, amsthm, slashed}
\usepackage{amssymb}
\usepackage{mathabx}
\usepackage{rotating}
\usepackage{mathrsfs}
\usepackage{epsfig}
\usepackage{tikz-cd}
\usepackage{verbatim}
\usepackage{amsmath, amssymb, amsthm, setspace, hyperref}
\usepackage{rotating}
\usepackage{graphicx}
\usepackage{accents}
\usepackage{enumitem}
\setlist{nolistsep}

\newtheorem{theorem}{Theorem}[section]

\newtheorem*{conjecture*}{Conjecture}
\newtheorem{corollary}[theorem]{Corollary}
\newtheorem*{theorem*}{Theorem}
\newtheorem*{corollary*}{Corollary}
\newtheorem*{nott*}{Notation}
\newtheorem{proposition}[theorem]{Proposition}
\newtheorem{lemma}[theorem]{Lemma}

\theoremstyle{definition}
\newtheorem{definition}[theorem]{Definition}
\newtheorem{example}[theorem]{Example}
\newtheorem{remark}[theorem]{Remark}

\newtheorem{nott}[theorem]{Notation}
\newtheorem{setup}[theorem]{Setup}

\newcommand{\id}{\text{id}}
\newcommand{\Mor}{\text{Mor}}
\newcommand{\ov}{\overline}
\newcommand{\si}{\Sigma}
\newcommand{\Spec}{\text{Spec }}
\newcommand{\Bl}{\text{Bl}}
\newcommand{\Proj}{\text{Proj}}
\newcommand{\tra}{\text{tr}}
\newcommand{\Frob}{\text{Frob}}
\newcommand{\pr}{\text{pr}}
\newcommand{\tr}{\vartriangle}

\newtheorem*{proposition*}{Proposition}
\newtheorem*{cor*}{Corollary}
\newtheorem*{question*}{Question}

\def\FF{{\mathcal F}}

\def\A{{\mathbb A}}

\def\C{{\mathbb C}}

\def\Q{{\mathbb Q}}

\def\P{{\mathbb P}}

\def\Z{{\mathbb Z}}

\DeclareMathAlphabet\mathbfcal{OMS}{cmsy}{b}{n}

\title{A moduli scheme parametrizing blowups of smooth projective surfaces}

\author{Monica Marinescu}
\date{}

\begin{document}

\begin{abstract}\sloppy
We construct a moduli scheme $F[n]$ that parametrizes tuples $(S_1, S_2, \dots, S_{n+1}, p_1, p_2, \dots, p_n)$ in which $S_1$ is a fixed smooth surface over $\Spec R$ and $S_{i+1}$ is the blowup of $S_i$ at a point $p_i$, $\forall 1\leq i\leq n$. We show that this moduli scheme is smooth and projective. We prove that $F[n]$ has smooth divisors $D_{i,j}^{(n)}$, $\forall 1\leq i<j\leq n$, which correspond to tuples that map $p_j\mapsto p_i$ under the projection morphism $S_j\to S_i$. When $R=k$ is an algebraically closed field, we demonstrate that the Chow ring $\A^*(F[n])$ is generated by these divisors over $\A^*(S_1^n)$. We end by giving a precise description of $\A^*(F[n])$ when $S_1$ is a complex rational surface.
\end{abstract}
\maketitle







\section{Definition of the moduli problem}
\label{functor}

Let $S_1$ be a fixed smooth projective surface over a ring $R$. Our goal is to construct a space that parametrizes surfaces obtained through a series of $n$ ordered blowups of the base surface $S_1$. More specifically, we focus our attention on the following objects:
\[(S_1,S_2,\dots, S_{n+1}, p_1, p_2,\dots, p_{n}),\]
where $p_{i}\in S_{i}$ and $S_{i+1}$ is the blowup of $S_{i}$ at $p_{i}$, for all $1\leq i\leq n$. We define formally the functor $\FF[S_1,n]$ as follows:

\begin{definition}
Let $S_1$ be a smooth projective surface over a ring $R$ and $n\geq 0$ an integer. Consider the contravariant functor:
\[\FF[n]=\FF[S_1,n]:\it{Sch}(R)\to\it{Sets}\]
defined as follows: 
 
 
$\bullet$ For any $R$-scheme $B$, an object in $\FF[S_1,n](B)$ is a tower of morphisms:

\[\begin{tikzcd}
\si_{B,n+1}\arrow[r, "\pi_{n+1}"] &\si_{B,n}\arrow[r, "\pi_{n}"] &\si_{B,n-1}\arrow[r, "\pi_{n-1}"] &\dots\arrow[r] &\si_{B,2}\arrow[r, "\pi_{2}"] &\si_{B,1}=B\times S_1\arrow[d, "\pi_{1}"]\\
 &&&&&B,\arrow[u, "p_{1}", shift left=1ex, bend left]\arrow[ul, "p_{2}", bend left]\arrow[ulll, "p_{n-1}", bend left]\arrow[ullll, "p_{n}", bend left]
\end{tikzcd}\]
such that the following conditions are satisfied:
\begin{enumerate}
\item[(1)] $\pi_1=\text{pr}_1$ is the projection onto the first factor;
\item[(2)] for each $1\leq i\leq n$, the morphism $p_{i}:B\to \si_{B,i}$ is a section of the composed map $\si_{B,i}\to B$;
\item[(3)] for each $1\leq i\leq n$, the morphism $\pi_{i+1}:\si_{B,i+1}\to\si_{B,i}$ is the blowup of $\si_{B,i}$ along the locus $p_{i}(B)$.
\end{enumerate}

Shortened notations for such a family are $\si_{B,\leq n+1}\to B$ and $(\si_{B,i},\pi_i,p_i)_{i=1}^n$. Notice that although $\si_{n+1}$ is not included in the latter notation, this scheme is uniquely defined by the data given.

$\bullet$ For every $R$-scheme $B$, $\FF[n](B)$ is the set of all families over $B$, up to isomorphism. An isomorphism between two families is defined as:
$$\Theta=(\theta_i)_{i=1}^n : (\si_{B,i},\pi_i,p_i)_{i=1}^n \xrightarrow{\sim} (\si_{B,i}',\pi_i',p_i')_{i=1}^n,$$
where, $\forall 1\leq i\leq n$, $\theta_i:\si_i\to \si_i'$ is a scheme isomorphism that commutes with the maps of the two families. 

$\bullet$ Let $B_1,B_2$ be $R$-schemes and $f:B_1\to B_2$ an $R$-morphism between the two schemes. There exists a natural contravariant map $\FF(f): \FF[n](B_2)\to \FF[n](B_1)$ that assigns to each family over $B_2$ a corresponding family over $B_1$ by pulling back the schemes $\si_{B_2,i}$ and the sections $p_i$ along $f$. The resulting tower of morphisms over $B_1$ is a valid object in $\FF[n](B_1)$ (this is an immediate application of Lemma~\ref{technical} from Section~\ref{moduli-space}).

$\bullet$ The identity map $\id :B\to B$ on any $R$-scheme corresponds to the identity map on sets $\FF(\id)=\id : \FF[n](B)\to \FF[n](B)$.

$\bullet$ Let $B_1, B_2, B_3$ be $R$-schemes. Let $f:B_1\to B_2$ and $g:B_2\to B_3$ be $R$-morphisms. Then $\FF(g\circ f)=\FF(f)\circ \FF(g)$. This follows from the uniqueness of pullbacks (up to isomorphism).

\end{definition}

\begin{remark}
Let $B$ be a $R$-scheme and $\si_{B,\leq n+1}\rightarrow B$ a family in $\FF[n](B)$. For any point $x\in B$, the fiber over $x$ is a sequence $S_{n+1}\to\dots \to S_1$, where $S_{i+1}$ is the blowup of $S_{i}$ at some point $p_{i}$, $\forall 1\leq i\leq n$. Therefore, this fiber corresponds to a tuple $(S_1,\dots, S_{n+1}, p_1,\dots, p_{n})$ as the ones introduced in the beginning of the section.
\end{remark}

\begin{remark}
For every integer $n\geq 0$, there exists a natural forgetful map $\FF[n+1]\to \FF[n]$ which sends families $\si_{B,\leq n+2}\to B$ in $\FF[n+1](B)$ to families $\si_{B,\leq n+1}\to B$ in $\FF[n](B)$, for all $R$-schemes $B$.
\end{remark}

\begin{remark}\label{nat-trans}
For every integer $n> 0$, there exists a natural transformation of functors $\FF[n]\to S_1^n$ which, for any $R$-scheme $B$, maps a family $(\si_{B,i},\pi_i,p_i)_{i=1}^n$ over $B$ to $(\ov{p_1}, \ov{p_2},\dots, \ov{p_{n}})$, where $\ov{p_i}\in S_1(B)$ is the composition map $B\xrightarrow{p_i} \si_{B,i}\to \dots\to \si_{B,1}=B\times S_1\xrightarrow{\pr_2} S_1$.
\end{remark}

\section{Construction of the moduli scheme}
\label{moduli-space}

In this section we prove that the functor $\FF[n]$ has a fine moduli scheme, which we denote by $F[n]$. We show that $F[n]$ and all the schemes in its universal family are smooth and projective over $\Spec R$. 

Before we start working towards our main result, note that we reserve the following notation for the universal family over $F[n]$:
\[\begin{tikzcd}
\si_{n+1,n+1}\arrow[r, "\pi_{n+1,n+1}"] &\si_{n+1,n}\arrow[r, "\pi_{n+1,n}"] &\dots\arrow[r] &\si_{n+1,2}\arrow[r, "\pi_{n+1,2}"] &\si_{n+1,1}=F[n]\times S_1\arrow[d, "\pi_{n+1,1}"]\\
 &&&&F[n],\arrow[u, "\sigma_{n+1,1}", shift left=2ex, bend left]\arrow[ul, "\sigma_{n+1,2}", bend left]\arrow[ulll, "\sigma_{n+1,n}", bend left]
\end{tikzcd}\]

\begin{example}\label{base-case} We start by constructing $F[0]$, $F[1]$ (which represent the functors $\FF[0]$ and $\FF[1]$, respectively), and their universal families. It is easy to see that $F[0]=\Spec R$, since every scheme $B\in Sch(R)$ comes equipped with a structure morphism to $\Spec R$. By definition, $\si_{1,1}=\Spec R\times S_1= S_1$:

\[\begin{tikzcd}[row sep=small, column sep=small]
\si_{B,1}\arrow[dr, phantom, "\ulcorner", very near start]\arrow[d,"\pr_{1}"']\arrow[r]&\si_{1,1}\arrow[d,"\pr_{1}"{name=U}] & & B\times S_1\arrow[dr, phantom, "\ulcorner", very near start]\arrow[r]\arrow[d,"\pr_{1}"'{name=V}]&S_1\arrow[d,"\pr_1"] \\
B\arrow[r]&F[0] & &B\arrow[r] &\Spec R.
\arrow[r, phantom,"\cong"', from=U, to=V]
\end{tikzcd}\]

Next, we show that $F[1]\cong S_1$. Intuitively, every object in this moduli space is a triple $(S_1,S_2, p_1)$ where $S_2=\Bl_{p_1}S_1$, so it is uniquely identified by the point $p_1\in S_1$. More concretely, we need to show that $\FF[1](B)\cong \text{Mor}(B,S_1)$. The equivalence goes as follows: for a family over $B$ like in the figure below, the corresponding map $B\to S_1$ is $f=\pr_2\circ p_1$. Conversely, given a morphism $f:B\to S_1$, we obtain a section $p_1=\id\times f:B\to \si_{B,1}=S_1\times B$ and $\si_{B,2}$ is the blowup of $\si_{B,1}$ along this section:
\[\begin{tikzcd}[row sep=small, column sep=small]
&\si_{B,2}\arrow[d, "\pi_{2}"']\\
S_1&\si_{B,1}=B\times S_1\arrow[d, "\pi_{1}"']\arrow[l, "\pr_2"']\\
&B.\arrow[u, "p_{1}=\id\times f"', shift right=2ex, bend right]\arrow[ul, "f=\pr_2\circ p_1"]
\end{tikzcd}\]

The top scheme in the universal family over $F[1]$ is $\si_{2,2}=\Bl_{\vartriangle}(S_1\times S_1)$. This follows immediately from the figure above, considering the special case where $B=F[1]\cong S_1$ and $f=\id:S_1\to S_1$:
\[\begin{tikzcd}[row sep=small, column sep=small]
\si_{2,2}\arrow[d,"\pi_{2,2}"] & & \Bl_{\tr}(S_1\times S_1)\arrow[d,"\text{bl}_{\tr}"]\\
\si_{2,1}\arrow[d, "\pi_{2,1}"] &\cong &S_1\times S_1\arrow[d,"\pr_1"]\\
F[1]\arrow[u, "p_{2,1}", shift left=2ex, bend left] &  &S_1.\arrow[u, "\tr",  shift left=2ex, bend left]
\end{tikzcd}\]
\end{example}

Before we give the general construction of the moduli scheme corresponding to the functor $\FF[n]$, we state and prove the following lemmas, which will be the key ingredients in our proof:

\begin{lemma}\label{technical}
Let $A, B, \si_{B,1}$ be schemes over $R$. Let $\pi:\si_{B,1}\to B$ be a smooth morphism with a section $\sigma:B\to\si_{B,1}$, and let $\si_{B,2}$ be the blowup of $\si_{B,1}$ along the locus $\sigma(B)$. Given $f:A\to B$ an $R$-morphism, let $\si_{A,1}$ and $ \sigma^*$ be the pullbacks along $f$ of $\si_{B,1}$ and $\sigma$, respectively. The following statements hold:
\begin{enumerate}
\item[(i)] The composed morphism $\si_{B,2}\to B$ is smooth.
\item[(ii)] The blowup of $\si_{A,1}$ along the locus $\sigma^*(A)$, denoted by $\si_{A,2}$, is the pullback of $\si_{B,2}$ along $f$.
\end{enumerate}
\[
\begin{tikzcd}
\si_{A,2}\arrow[dr, phantom, "\ulcorner", very near start]\arrow[r]\arrow[d] &\si_{B,2}\arrow[d]\\
\si_{A,1}\arrow[dr, phantom, "\ulcorner", very near start]\arrow[r]\arrow[d,"\pi^*"']&\si_{B,1}\arrow[d, "\pi"]\\
A\arrow[u, "\sigma^*", shift left=1.5ex, bend left]\arrow[r, "f"] &B\arrow[u, "\sigma"', shift right=1ex, bend right]
\end{tikzcd}\]
\end{lemma}

\begin{proof}
(i) We show that the morphism $\si_{B,2}\to B$ is smooth by proving that it is flat, locally of finite presentation, and has smooth fibers (see~\cite{stacks}, \href{https://stacks.math.columbia.edu/tag/02K5}{Tag 02K5}).
Affine locally, on the level of rings, we are given a smooth morphism $\pi: T\to T'$ and $\sigma:T'\twoheadrightarrow T$ a section of $\pi$. Let $I=\ker(\sigma)\subset T'$. The blowup of $\Spec T'$ along $I$ is $\Bl_I(\Spec T')=\Proj(T'\oplus I\oplus I^2\oplus \dots)$.

The map $\sigma$ is a section of the smooth morphism $\pi$, which means that $I$ is a regular ideal. Since $I$ is regular, then $\Bl_I(\Spec T')$ is of finite presentation over $\Spec T$ (see~\cite{stacks}, \href{https://stacks.math.columbia.edu/tag/0BIQ}{Tag 0BIQ}).

To show that $\si_{B,2}\to B$ is flat, it is enough to prove that $(T'\oplus I\oplus I^2\oplus\dots)$ is flat over $T$. We know $T'$ is flat over $T$ because $T\to T'$ is smooth. We show inductively that $I^n$ is flat over $T$. By the following short exact sequence, it is enough to prove that $T'/I^n$ is $T$-flat:
\[0\to I^{n}\to T'\to T'/I^{n}\to 0.\]

When $n=1$, $T'/I\cong T$, so the claim is true. For the inductive step, consider another short exact sequence:
\[0\to I^{n-1}/I^n\to T'/I^n\to T'/I^{n-1}\to 0.\]

Since $I$ is a regular ideal, then $I^{n-1}/I^n$ is a locally free finite $T$-module, hence it is flat. By the induction hypothesis, $T'/I^{n-1}$ is flat over $T$. Putting these two facts together, we conclude that $T'/I^n$ is flat over $T$, completing the induction step.

Lastly, we show that the morphism $\si_{B,2}\to B$ has smooth fibers. By part (ii) below, the fiber over every point $x\in B$ is:
\[\begin{tikzcd}
V'=\Bl_{x}V\arrow[r] & V\arrow[r] & \Spec k(x)\arrow[l, twoheadrightarrow, bend left,"x"],
\end{tikzcd}\]
where $k(x)$ is the residue field of $x$. The scheme $V$ is smooth over $k(x)$, hence $V'=\Bl_xV$ is also smooth over $k(x)$. With this, we conclude that the morphism $\si_{B,2}\to B$ is smooth.

\item[(ii)]


We work affine locally, on the level of rings, where we have the following figure:
\[
\begin{tikzcd}
S'=S\arrow[d,twoheadrightarrow, "\sigma_S"', shift right=1ex, bend right]\otimes_TT'&T'\arrow[l]\arrow[d,twoheadrightarrow,  "\sigma", shift left=1ex, bend left]\\
S\arrow[u,"\pi_S"']&T\arrow[l]\arrow[u, "\pi"]
\end{tikzcd}\]

Let $\pi_S$ and $\sigma_S$ be the pullbacks of $\pi$ and $\sigma$, respectively. Let $I_S=\ker(\sigma_S)\subset S'$. The blowup of $\Spec S'$ along $I_S$ is $\Bl_{I_S}(\Spec S')=\Proj(S'\oplus I_S\oplus I_S^2\oplus\dots)$. Our claim boils down to showing that $I_S^n=I^n\otimes_{T'}S'=I^n\otimes_TS$. This fact follows from the figure below. Note that $T'/I^n$ is flat over $T$, so tensoring the top short exact sequence with $S$ preserves the exactness of the resulting sequence:
\[\begin{tikzcd}
0\arrow[r]&I^n\otimes_T S\arrow[r]\arrow[d]&T'\otimes_TS\arrow[r]\arrow[d,"\cong"]&T'/I^n\otimes_TS\arrow[r]\arrow[d,"\cong"]&0\\
0\arrow[r]&I_S^n\arrow[r]&S'\arrow[r]&S'/I_S^n\arrow[r]&0.
\end{tikzcd}\]\qedhere
\end{proof}

\begin{lemma}\label{technical2}
Let $A, B, C$ be $R$-schemes. Let $A\xrightarrow{h}C$ be an $R$-morphism that factors as $A\xrightarrow{f}B\xrightarrow{g}C$. Given any $R$-morphism $V_C\xrightarrow{v_C}C$, let $V_B\xrightarrow{v_B}B$ be its pullback along $g$, and $V_A\xrightarrow{v_A}A$ be its pullback along $h$. There exists a unique map $f':V_A\to V_B$ which makes the top triangle commutative and the left square cartesian.
\[\begin{tikzcd}[row sep=small, column sep=small]
V_A\arrow[dd,"v_A"]\arrow[rd, dashed, "\exists ! f'"]\arrow[rr,"h'"] & &V_C\arrow[dd,"v_C"]\\
& V_B\arrow[]{ru}[name=N]{g'}\arrow[dd, "v_B"' near start] & \\
A\arrow[rr, "\hspace{30pt}h"]\arrow[rd, "f"] & &C\\
& B\arrow[ru, "g"] &
\end{tikzcd}\]
\end{lemma}

\begin{proof}
Let $V_A'=V_B\times_B A$ with $f':V_A'\to V_B$ the associated map. Since $V_B=V_C\times_C B$ and $h=g \circ f$, then $V_A'=V_C\times_C A$. By the uniqueness of pullbacks, we conclude that $V_A=V_A'$ and $f'$ is the unique map in question.\qedhere

\end{proof}

\begin{theorem}\label{main2}
The functor $\FF[n]$ has a fine moduli scheme $F[n]$. Moreover, the following is true:
\begin{enumerate}
\item[(a)] For every $n\geq 0$, $F[n]$ and all the schemes $\si_{n+1,1}$, $\dots$, $\si_{n+1,n+1}$ in its universal family are smooth and projective over $\Spec R$;
\item[(b)] For every $n\geq 1$, the top scheme $\si_{n+1,n+1}$ in the universal family over $F[n]$ can be identified as:
\[\si_{n+1,n+1}\cong \text{Bl}_{\vartriangle}(F[n]\times_{F[n-1]}F[n]),\]
where the cartesian product is induced by the forgetful map $F[n]\to F[n-1]$;
\item[(c)] For every $n\geq 0$, the top scheme $\si_{n+1,n+1}$ in the universal family of $F[n]$ is the moduli scheme representing the functor $\FF[n+1]$. Under this identification, the map 
\[\Pi_{n+1}=\pi_{n+1,1}\circ\dots\circ\pi_{n+1,n+1}:\si_{n+1,n+1}\cong F[n+1]\to F[n]\]
corresponds to the forgetful functor $\FF[n+1]\to \FF[n]$. 
\end{enumerate}
\end{theorem}

\begin{proof}

We prove these statements inductively over $n$. The base case is covered in Example~\ref{base-case}. For the inductive step, assume that the moduli scheme $F[k]$ exists, for all $k< n$. Assume $F[k], \si_{k+1,1},\dots, \si_{k+1,k+1}$ are all smooth and projective. Here are the steps in our proof:


\begin{enumerate}
\item[(i)] Let $W=\si_{n,n}$ be the top scheme in the universal family over $F[n-1]$. We construct a family over $W$ in $\FF[n](W)$, then show that $W$ is the fine moduli scheme corresponding to the functor $\FF[n]$ and that the family we just defined over $W$ is the universal family;
\item[(ii)] We show that $F[n], \si_{n+1,1},\dots,\si_{n+1,n+1}$ are smooth and projective, and that $\si_{n+1,n+1}$ is isomorphic to $\Bl_{\tr}(F[n]\times_{F[n-1]}F[n])$;
\item[(iii)] We prove that the map $\Pi_{n}:\si_{n,n}\cong F[n]\to F[n-1]$ corresponds to the forgetful functor $\FF[n]\to \FF[n-1]$.
\end{enumerate}

The first step is to construct the family over $W$ in $\FF[n](W)$. Let $\Pi_{n}=\pi_{n,1}\circ\dots\circ\pi_{n,n}:W=\si_{n,n}\to F[n-1]$ be the composed morphism and $\si_{W,\leq n}\to W$ be the pullback of the universal family $\si_{n,\leq n}\to F[n-1]$ along this map $\Pi_n$ (see figure below). In particular, notice that $\si_{W,n}\cong W\times_{F[n-1]}W$. Let $p_{W,n}:=\tr:W\to \si_{W,n}=W\times_{F[n-1]}W$ be the diagonal embedding and $\si_{W, n+1}=\Bl_{\tr}(W\times_{F[n-1]}W)$:

\[
\begin{tikzcd}[row sep=small, column sep=small]
\si_{W,n+1}=\Bl_{\tr}(W\times_{F[n-1]}W)\arrow[d] &\\
\si_{W,n}=W\times_{F[n-1]}W\arrow[dr, phantom, "\ulcorner", very near start]\arrow[r]\arrow[d] &\si_{n,n}=W\arrow[d]\arrow[dd,  "\Pi_{n}", shift left=3ex, bend left]\\
\si_{W,\leq n-1}\arrow[dr, phantom, "\ulcorner", near start]\arrow[r]\arrow[d]&\si_{n,\leq n-1}\arrow[d]\\
W\arrow[uu, "p_{W,n}=\tr", shift left=3ex, bend left]\arrow[r, "\Pi_{n}"] &F[n-1].
\end{tikzcd}
\]

The resulting tower of morphisms $\si_{W,\leq n+1}\to W$ is indeed a family in $\FF[n](W)$, as a result of Lemma~\ref{technical}. We claim that $W$ is the fine moduli space for $\FF[n]$, and that the family constructed above is the universal family over $F[n]$. We prove this statement by using this family over $W$ to build the correspondence $\FF[n](B)\cong \Mor(B,W)$, for any $R$-scheme $B$.

We start by constructing a functor map:
\[C_1:\FF[n]\to \Mor(-,W).\]

Let $B$ be an $R$-scheme and $\si_{B,\leq n+1}\to B$ a family in $\FF[n](B)$. We need to associate a morphism $B\to W$ to this family. The truncated family $\si_{B,\leq n}\to B$ is an element of $\FF[n-1](B)$, so it corresponds uniquely to a morphism $f_{n-1}:B\to F[n-1]$ that gives the figure below. Now, recall that the family $\si_{B,\leq n+1}\to B$ comes equipped with a section $p_n:B\to \si_n$, so we obtain the desired map $f_W:B\to W$ by composing $B\xrightarrow{p_{n}}\si_{B,n}\xrightarrow{\pr_2} W$:

\[
\begin{tikzcd}[row sep=small, column sep=small]
\si_{B,n}\arrow[dr, phantom, "\ulcorner", very near start]\arrow[r, "\pr_2"]\arrow[d] &\si_{n,n}=W\arrow[d]\\
\si_{B,\leq n-1}\arrow[dr, phantom, "\ulcorner", very near start]\arrow[r]\arrow[d]&\si_{n,\leq n-1}\arrow[d]\\
B\arrow[uu, "p_{n}", shift left=3ex, bend left]\arrow[r,"f_{n-1}"] &F[n-1].
\end{tikzcd}
\]

Second, we construct a functor map:
\[C_2:\Mor(-,W)\to \FF[n].\]

Let $B$ be an $R$-scheme and $f:B\to W$ an $R$-morphism. We obtain a corresponding family in $\FF[n](B)$ by pulling back the family $\si_{W,\leq n+1}\to W$ along $f$. The resulting tower of morphisms is indeed a family in $\FF[n](B)$, as a result of Lemma~\ref{technical}:
\[
\begin{tikzcd}[row sep=small, column sep=small]
\si_{B,\leq n+1}\arrow[dr, phantom, "\ulcorner", very near start]\arrow[r]\arrow[d] &\si_{W,\leq n+1}\arrow[d]\\
B\arrow[u, "p^*_{W,\bullet}", shift left=1ex, bend left]\arrow[r] &W.\arrow[u, "p_{W,\bullet}"', shift right=1ex, bend right]
\end{tikzcd}
\]

Now we want to show that the maps $C_1$ and $C_2$ are inverses of each other. Say we start with a morphism $f:B\to W$. We construct a family over $B$ by pulling back $\si_{W,\leq n+1}\to W$ along $f$, then use this family to obtain a corresponding morphism $f_W:B\to W$:

\[
\begin{tikzcd}[row sep=small, column sep=small]
\si_{B,n+1}\arrow[dr, phantom, "\ulcorner", very near start]\arrow[r]\arrow[d] &\si_{W,n+1}\arrow[d]\\
\si_{B,n}\arrow[dr, phantom, "\ulcorner", very near start]\arrow[r]\arrow[d] &\si_{W,n}\arrow[dr, phantom, "\ulcorner", very near start]\arrow[d]\arrow[r] &\si_{n,n}\cong W\arrow[d]\\
\si_{B,\leq n-1}\arrow[dr, phantom, "\ulcorner", very near start]\arrow[r]\arrow[d]&\si_{W,\leq n-1}\arrow[dr, phantom, "\ulcorner", very near start]\arrow[d]\arrow[r] &\si_{n,\leq n-1}\arrow[d]\\
B\arrow[uu, "p_{n}", shift left=3ex, bend left]\arrow[r, "f"] &W\arrow[r] &F[n-1].
\end{tikzcd}
\]

We prove that $f=f_W$, which concludes that $C_1\circ C_2=\id$:
\begin{align*}
[B\xrightarrow{f_W}W] &= [B\xrightarrow{p_{n}}\si_{B,n}\to\si_{W,n}\to \si_{n,n}]\\
&= [B\xrightarrow{f} W\xrightarrow{p_{W,n}}\si_{W,n}\to\si_{n,n}]\\
&= [B\xrightarrow{f} W\xrightarrow{\tr}W\times_{F[n-1]}W\xrightarrow{\pr_2}W]\\
&=[B\xrightarrow{f} W].
\end{align*}

Say we start with a family over $B$ in $\FF[n](B)$, denoted as $\si_{B,\leq n+1}\to B$. As before, we obtain a corresponding morphism $f_W:B\to W$. We want to show that if we pull back $\si_{W, \leq n+1}\to W$ along $f_W$, we recover the same family we started with. Since the truncated family $\si_{B,\leq n}\to B$ is an object in $\FF[n-1](B)$, it corresponds uniquely to a morphism $f: B\to F[n-1]$ which gives this figure:


\[
\begin{tikzcd}[row sep=small, column sep=small]
\si_{B,n+1}\arrow[d]\\
\si_{B,n}\arrow[dr, phantom, "\ulcorner", very near start]\arrow[r, "\pr_2"]\arrow[d] &\si_{n, n}\arrow[d]\\
\si_{B,\leq n-1}\arrow[dr, phantom, "\ulcorner", very near start]\arrow[r]\arrow[d] &\si_{n,\leq n-1}\arrow[d]\\
B\arrow[r,"f"]\arrow[uu, "p_n", shift left=2ex, bend left] &F[n-1].
\end{tikzcd}
\]

We first claim that the map $B\xrightarrow{f} F[n-1]$ factors as $B\xrightarrow{f_W}W\xrightarrow{\Pi_n}F[n-1]$. This is true because of the following equivalence of maps:
\begin{align*}
[B\xrightarrow{f} F[n-1]] &= [B\xrightarrow{\id}B\xrightarrow{f} F[n-1]]\\
&= [B\xrightarrow{p_{n}}\si_{B,n}\to B\xrightarrow{f} F[n-1]]\\
&= [B\xrightarrow{p_{n}}\si_{B,n}\xrightarrow{\pr_2}\si_{n,n}\xrightarrow{\Pi_n} F[n-1]]\\
&= [B\xrightarrow{f_W} W\xrightarrow{\Pi_n} F[n-1]].
\end{align*}

Since the morphism $B\xrightarrow{f} F[n-1]$ factors as $B\xrightarrow{f_W}W\xrightarrow{\Pi_n}F[n-1]$, we can use Lemma~\ref{technical2} repeatedly, from the bottom up, to recover the maps $\si_i\to\si_{W,i}$, for all $i\leq n$, which make every rectangle in the figure below cartesian, and every triangle commutative:
 \[\begin{tikzcd}[row sep=small, column sep=small]
\si_{B,n+1}\arrow[dd] & &\\
& \si_{W,n+1}\arrow[dd]& \\
\si_{B,\leq n}\arrow[rr]\arrow[dd]\arrow[rd] & & \si_{n,\leq n}\arrow[dd]\\
& \si_{W,\leq n}\arrow[dd]\arrow[ru]& \\
B\arrow[rr, "\hspace{30pt}f"]\arrow[rd, "f_W"] & & F[n-1]\\
& W\arrow[ru, "\Pi_n"] &
\end{tikzcd}
\]

Lastly, we need to show that the section $p_{n}:B\to \si_{B,n}$ is the pullback along $f_W$ of the diagonal embedding $p_{W,n}=\tr:W\to\si_{W,n}=W\times_{F[n-1]}W$. This follows from the figure below: by definition, $f_W$ is the composition $B\xrightarrow{p_n}\si_{B,n}\to \si_{W,n}=W\times_{F[n-1]}W\to \si_{n,n}=W$, so all maps in the figure commute as expected, and $p_{n}$ is indeed the pullback of $p_{W,n}=\tr$. By Lemma~\ref{technical}, we obtain that $\si_{B,n+1}$ is the pullback of $\si_{W,n+1}$, and the proof of (i) is complete:

\[\begin{tikzcd}[row sep=small, column sep=small]
B\arrow[rd, "p_{n}"]\arrow[rdd, "\id", bend right]\arrow[rrd, "f_W\times f_W", bend left]\arrow[rrrd, "f_W", bend left] & & & \\
&\si_{B,n}\arrow[r]\arrow[d] &W\times_{F[n-1]}W\arrow[r]\arrow[d] &W\arrow[d]\\
&B\arrow[r,"f_W"] &W\arrow[u, "\tr", bend left]\arrow[r] &F[n-1].
\end{tikzcd}
\]

For part (ii) we show inductively that $F[n]$ and all the schemes in its universal family are smooth and projective over $\Spec R$. As the base case, recall that $F[0]\cong \Spec R$ and $\si_{1,1}\cong S_1$ have this property. Inductively, assume that $F[n-1]$ and all the schemes in its universal family are smooth and projective over $\Spec R$. From the arguments above, $F[n]$ can be identified with the top scheme $\si_{n,n}$ over $F[n-1]$, hence it is also smooth and projective. The first scheme in the universal family over $F[n]$ is, by definition, $\si_{n+1,1}\cong F[n]\times S_1$. By construction, all the maps $\si_{n+1,i}\to F[n]$ are smooth (this is an application of Lemma~\ref{technical}), which means the corresponding sections $\sigma_{n+1,i}:F[n]\to \si_{n+1,i}$ are regular embeddings. Going up in the tower of morphisms, we can conclude step by step that $\si_{n+1,2}, \dots, \si_{n+1,n+1}$ are smooth and projective, since each of them is obtained by blowing up a smooth projective scheme along a smooth projective subscheme. This concludes claim (ii).
  
Lastly, we need to show that $\Pi_{n}:\si_{n,n}\cong F[n]\to F[n-1]$ corresponds to the forgetful functor $\FF[n]\to \FF[n-1]$. For any $R$-scheme $B$ and any morphism $f:B\to F[n]$, let $\si_{B,\leq n+1}\to B$ be its corresponding family in $\FF[n](B)$. The map $\Pi_{n}\circ f: B\to F[n-1]$ corresponds to the truncated family $\si_{B,\leq n}\to B$ in $\FF[n-1](B)$, concluding part (iii):

\[
\begin{tikzcd}
\si_{n+1}\arrow[dr, phantom, "\ulcorner", very near start]\arrow[r]\arrow[d] &\si_{n+1,n+1}\arrow[d]\\
\si_{\leq n}\arrow[dr, phantom, "\ulcorner", very near start]\arrow[r]\arrow[d] &\si_{n+1,\leq n}\arrow[dr, phantom, "\ulcorner", very near start]\arrow[d]\arrow[r] &\si_{n,\leq n}\arrow[d]\\
B\arrow[r, "f"] &F[n]\arrow[r,"\Pi_{n}"] &F[n-1].
\end{tikzcd}
\]
\end{proof}

\begin{remark} As a consequence of the construction in the proof of Theorem~\ref{main2}, we obtain the following ascending ladder, where each square is cartesian:

\[\begin{tikzcd}
\dots\arrow[r]& \si_{5,5}\arrow[d,"\pi_{5,5}"']\\
\dots\arrow[r]& \si_{5,4}\arrow[dr, phantom, "\ulcorner", very near start]\arrow[r]\arrow[d, "\pi_{5,4}"'] &\si_{4,4}\arrow[d,"\pi_{4,4}"']\\
\dots\arrow[r]& \si_{5,3}\arrow[dr, phantom, "\ulcorner", very near start]\arrow[r]\arrow[d," \pi_{5,3}"'] &\si_{4,3}\arrow[dr, phantom, "\ulcorner", very near start]\arrow[r]\arrow[d, "\pi_{4,3}"'] &\si_{3,3}\arrow[d, "\pi_{3,3}"']\\
\dots\arrow[r]& \si_{5,2}\arrow[dr, phantom, "\ulcorner", very near start]\arrow[r]\arrow[d, "\pi_{5,2}"'] &\si_{4,2}\arrow[dr, phantom, "\ulcorner", very near start]\arrow[r]\arrow[d, "\pi_{4,2}"'] &\si_{3,2}\arrow[dr, phantom, "\ulcorner", very near start]\arrow[r]\arrow[d, "\pi_{3,2}"'] &\si_{2,2}\arrow[d, "\pi_{2,2}"']\\
\dots\arrow[r]&\si_{5,1}\arrow[dr, phantom, "\ulcorner", very near start]\arrow[r]\arrow[d, "\pi_{5,1}"']&\si_{4,1}\arrow[dr, phantom, "\ulcorner", very near start]\arrow[r]\arrow[d,"\pi_{4,1}"']&\si_{3,1}\arrow[dr, phantom, "\ulcorner", very near start]\arrow[r]\arrow[d, "\pi_{3,1}"']&\si_{2,1}\arrow[dr, phantom, "\ulcorner", very near start]\arrow[r]\arrow[d, "\pi_{2,1}"']&\si_{1,1}\arrow[d, "\pi_{1,1}"']\\
\dots\arrow[r]&F[4]\arrow[r,"\Pi_{4}"'] &F[3]\arrow[r, "\Pi_{3}"']&F[2]\arrow[r, "\Pi_{2}"']&F[1]\arrow[r, "\Pi_{1}"']&F[0].
\end{tikzcd}\]

In particular, since $\si_{n,n}\cong F[n]$, for all $n\geq 1$, we obtain the following identification:
\begin{align}\label{si-n-i}
\si_{n+1,i}\cong F[n]\times_{F[i-1]}F[i], \forall 1\leq i\leq n+1.
\end{align}
\end{remark}

In light of Equation~\ref{si-n-i}, we look back to the universal family over $F[n]$ and give another description of the projection maps $\pi_{n+1,*}$ and the sections $\sigma_{n+1,*}$. Before we do so, we need to establish some notation:

\begin{nott}\label{k-points}
Let $B$ be an $R$-scheme. A family $(\si_{B,i}, \pi_i, p_i)_{i=1}^n$ in $F[n](B)$ will simply be denoted as $(p_1,\dots, p_n)$. Similarly, a $B$-point of $\si_{n+1,i}\cong F[n]\times_{F[i-1]} F[i]$ will be denoted by $(p_1,\dots, p_n; p_i')$, with the understanding that $(p_1,\dots, p_n)$ is the corresponding point in $F[n]$ and $(p_1,\dots, p_{i-1}, p_i')$ is the one in $F[i]$.
\end{nott}


\begin{proposition}\label{beh-maps}
Let $B$ be a $R$-scheme and $(\si_{n+1,i},\pi_{n+1,i},\sigma_{n+1,i})_{i=1}^{n}$ the universal family over the moduli scheme $F[n]$. Using Notation~\ref{k-points} above, the morphisms $\pi_{n+1,*}$ and $\sigma_{n+1,*}$ map $B$-points as follows:
\begin{enumerate}
\item[(a)] $\sigma_{n+1,i}:F[n](B)\to \si_{n+1,i}(B)\text{ maps }(p_1,\dots, p_{n})\mapsto (p_1,\dots, p_{n};p_i)$;
\item[(b)] $\pi_{n+1,i}:\si_{n+1,i}(B)\to \si_{n+1,i-1}(B)\text{ maps }(p_{1},\dots, p_{n};p'_{i})\mapsto (p_{1},\dots, p_{n}; \overline{p'_{i}})$, where $\overline{p'_{i}}:B\to \si_{B,i-1}$ is the image of $p'_{i}$ under the projection map $\si_{B,i}\to \si_{B,i-1}$. 
\end{enumerate}
\end{proposition}

\begin{proof}
(a) Let $f:F[n]\to F[i]$ be the projection morphism which maps $B$-points $(p_1,\dots, p_{n})\mapsto (p_1,\dots, p_i)$. We have the following diagram:
\[\begin{tikzcd}[row sep=small, column sep=small]
F[n]\arrow[rd, "\sigma_{n+1,i}"]\arrow[rdd, "\id", bend right]\arrow[rrd, "\tr\circ f", bend left] \\
&\si_{n+1,i}=F[n]\times_{F[i-1]}F[i]\arrow[r, "f\times \id"]\arrow[d, "\pr_1"] &\si_{i+1,i}=F[i]\times_{F[i-1]}F[i]\arrow[d,  "\pr_1"]\\
&F[n]\arrow[r,"f"] &F[i],\arrow[u, "\tr", bend left]
\end{tikzcd}\]
where the section $\sigma_{n+1,i}:F[n]\to\si_{n+1,i}$ is defined to be the unique morphism making the two triangles of the diagram commute. It is therefore clear that $\sigma_{n+1,i}$ maps a $B$-point $(p_1,\dots,p_{n})$ of $F[n]$ to $(p_1,\dots,p_{n};p_i)\in \si_{n+1,i}(B)$.

(b) We start by showing that the morphism $\pi_{n+1,n+1}: \si_{n+1,n+1}\to\si_{n+1,n}$ maps $B$-points $(p_1,\dots, p_{n},p_{n+1})\mapsto (p_1,\dots,p_{n};\ov{p_{n+1}})$, where $\ov{p_{n+1}}$ is the image of $p_{n+1}$ under the projection map $\si_{B,n+1}\to \si_{B,n}$. Since $\si_{n+1,n}=F[n]\times_{F[n-1]}F[n]$, it is enough to show that:
\begin{align*}
\pr_1\circ \pi_{n+1,n+1}:(p_1,\dots,p_{n+1})&\mapsto (p_1,\dots, p_{n-1}, p_n),\\
\pr_2\circ \pi_{n+1,n+1}:(p_1,\dots,p_{n+1})&\mapsto (p_1,\dots, p_{n-1}, \ov{p_{n+1}}).
\end{align*}

The map $\pr_1\circ\pi_{n+1,n+1}=\Pi_{n+1}:\si_{n+1,n+1}\to F[n]$ corresponds to the forgetful functor $\FF[n+1]\to \FF[n]$ (see Theorem~\ref{main2}), so the first claim above is true. To show the second identity, let 
$f_{n+1}:B\to F[n+1]$ be such a point and $f_{n}=\Pi_{n+1}\circ f_{n+1}:B\to F[n]$ be the `truncated` point. The maps $f_{n+1}$ and $f_n$ produce the following diagram, where $p_{n+1}$ is the pullback of $f_{n+1}$, and $p_1,\dots, p_{n}$ are the pullbacks of $\sigma_{n+1,1},\dots, \sigma_{n+1,n}$, respectively:

\[
\begin{tikzcd}[row sep=small, column sep=small]
B\arrow[dr, "p_{n+1}"]\arrow[drr, "f_{n+1}", bend left]\arrow[ddr, "p_{n}", bend right]\arrow[ddddr, "p_1", bend right]\arrow[dddddr, "\id"', bend right]& & \\
&\si_{B,n+1}\arrow[r]\arrow[d] &\si_{n+1,n+1}\cong F[n+1]\arrow[d, "\pi_{n+1,n+1}"]\\
&\si_{B,n}\arrow[r]\arrow[d] &\si_{n+1,n}\arrow[d]\arrow[r, "\pr_2"] &\si_{n,n}\cong F[n]\arrow[d]\\
&\dots\arrow[d]&\dots\arrow[d]&\dots\arrow[d]\\
&\si_{B,1}\arrow[r]\arrow[d]&\si_{n+1,1}\arrow[d]\arrow[r]&\si_{n,1}\arrow[d]\\
&B \arrow[r, "f_{n}"] &F[n]\arrow[r]\arrow[u, "\sigma_{n+1,1}"', shift right=2ex, bend right]\arrow[uuu, "\sigma_{n+1,n}"', shift right=3.1ex, bend right]&F[n-1].
\end{tikzcd}
\]

Notice the following maps are equivalent:
\begin{align*}
[B\xrightarrow{f_{n+1}} \si_{n+1,n+1}\xrightarrow{\pi_{n+1,n+1}}\si_{n+1,n}&\xrightarrow{\pr_2} \si_{n,n}]=\\
 &= [B\xrightarrow{p_{n+1}} \si_{B,n+1}\to\si_{n+1,n+1}\to \si_{n+1,n}\to\si_{n,n}]\\
 &= [B\xrightarrow{p_{n+1}} \si_{B,n+1}\to \si_{n}\to \si_{n+1,n}\to\si_{n,n}]\\
 &= [B\xrightarrow{\ov{p_{n+1}}} \si_{B,n}\to \si_{n+1,n}\to\si_{n,n}].
 \end{align*}
 
By the equivalence above, $\pr_2\circ \pi_{n+1,n+1}$ maps a $B$-point $(p_1,\dots, p_{n+1})$ in $\si_{n+1,n+1}$ to $(p_1,\dots, p_{n-1},\ov{p_{n+1}})$ in $\si_{n,n}$, and this concludes the proof of our initial statement.

The behavior of the more general map $\pi_{n+1,i}$ becomes clear from the following cartesian square, in which the bottom horizontal maps is induced by the forgetful map $F[n]\to F[i-1]$:
\[
\begin{tikzcd}
\si_{n+1,i}\cong F[n]\times_{F[i-1]}F[i]\arrow[r]\arrow[d,"\pi_{n+1,i}"]&\si_{i,i}\cong F[i]\arrow[d, "\pi_{i,i}"]\\
\si_{n,i-1}\cong F[n]\times_{F[i-2]}F[i-1]\arrow[r] &\si_{i,i-1}\cong F[i-1]\times_{F[i-2]}F[i-1].
\end{tikzcd}\]


\end{proof}

\begin{nott}
As a result of Proposition~\ref{beh-maps}, we change notation and denote the section $\sigma_{n+1,i}$ as $\tr_{i,n+1}$. We will use this notation throughout the rest of the paper.
 \end{nott}

\section{Divisors of the moduli scheme}
\label{divisors}

\begin{definition}
The moduli scheme $F[n]$ comes equipped with divisors $D^{(n)}_{i,j}, \forall 1\leq i<j\leq n$, which arise naturally from the construction outlined in the previous section. We start by defining the divisors $D^{(n)}_{1,n},\dots, D^{(n)}_{n-1,n}$. To do so, recall that for all $1\leq i\leq n-1$, $\si_{n,i+1}$ is obtained by blowing up the previous variety $\si_{n,i}$ along the locus $\tr_{i,n}(F[n-1])$. We define $D^{(n)}_{i,n}$ on $\si_{n,i+1}$ to be the exceptional divisor of this blowup:
\[\begin{tikzcd}
D_{i,n}^{(n)}\subset \Bl_{\tr_{i,n}}\si_{n,i}=\si_{n,i+1}\arrow[r] & \si_{n,i}\arrow[r] & F[n-1]\arrow[l, bend left,"\tr_{i,n}"].
\end{tikzcd}\]

By abuse of notation, the divisor $D^{(n)}_{i,n}$ on $F[n]=\si_{n,n}$ is the strict transform of the exceptional divisor coming from $\si_{n,i+1}$ in the tower of blowups $\si_{n,n}\to\si_{n,n-1}\to\dots\to\si_{n,i+1}$. 

We define the other divisors inductively as follows: given $D^{(n-1)}_{i,j}$ on $F[n-1]$, where $j\leq n-1$, let $D^{(n)}_{i,j}=D^{(n-1)}_{i,j}\times S_1$ be the divisor on $\si_{n,1}=F[n-1]\times S_1$. By abuse of notation, $D^{(n)}_{i,j}\subset F[n]=\si_{n,n}$ is defined as the strict transform of this divisor in the tower of blowups $\si_{n,n}\to\si_{n,n-1}\to\dots\to\si_{n,1}$.
\end{definition}

\begin{proposition}
For every $n\geq 1$, the natural projection map $F[n]\to S_1^n$ is a composition of blowups. Under this map, every divisor $D^{(n)}_{i,j}$ is mapped surjectively onto the diagonal $\tr_{i,j}$.
\end{proposition}

\begin{proof}
In the previous section, we constructed a tower of blowups:
\[F[n]=\si_{n,n}\xrightarrow{\pi_{n,n}} \si_{n,n-1}\xrightarrow{\pi_{n,n-1}}\dots\to\si_{n,2}\xrightarrow{\pi_{n,2}}\si_{n,1}=F[n-1]\times S_1.\]

Inductively, using that $F[0]\cong \Spec R$, we obtain a morphism $F[n]\to S_1^n$ that decomposes as a series of blowups. Given the behavior of the maps $\pi_{n,i}$, outlined in Proposition~\ref{beh-maps}, this map coincides with the natural projection morphism.

We show inductively on $n$ that the projection morphism maps surjectively any divisor $D^{(n)}_{i,j}$ onto the diagonal $\tr_{i,j}$. As the base case, recall that $F[2]$ is isomorphic to $Bl_{\tr_{1,2}}(S_1\times S_1)$, and $D^{(2)}_{1,2}$ is the exceptional divisor of this blowup. For the induction step, consider a divisor $D^{(n)}_{i,j}$ of the moduli space $F[n]$. If $j<n$, then, by construction, the morphism $\si_{n,n}\to\si_{n,1}=F[n-1]\times S_1$ maps $D_{i,j}^{(n)}\mapsto D_{i,j}^{(n-1)}\times S_1$, so the conclusion follows. If $j=n$, the divisor $D^{(n)}_{i,n}$ on $F[n]=\si_{n,n}$ is the strict transform of the exceptional divisor:
\[\begin{tikzcd}
D_{i,n}^{(n)}\subset \Bl_{\tr_{i,n}}\si_{n,i}=\si_{n,i+1}\arrow[r] & \si_{n,i}\arrow[r] & F[n-1]\arrow[l, bend left,"\tr_{i,n}"].
\end{tikzcd}\]

Given the behavior of the maps $\pi_{n,i}$, outlined in Proposition~\ref{beh-maps}, the projection morphism $F[n]=\si_{n,n}\to \si_{n,i}\to S_1^n$ maps $D^{(n)}_{i,n}\subset \si_{n,n}$ onto the blowup locus $\tr_{i,n}(F[n-1])\subset \si_{n,i}$, which is then mapped onto the diagonal $\tr_{i,n}\subset S_1^n$, and the proof is complete.
\end{proof}

\begin{proposition}\label{full}
(a) Let $1\leq i<j<n$, $1\leq k\leq n$. The divisor $D_{i,j}^{(n)}\subset \si_{n,k}$ is the inverse image of $D_{i,j}^{(n-1)}\subset F[n-1]$ under the projection map $\si_{n,k}\to F[n-1]$.

(b) Let $1\leq i<k\leq n$. The divisor $D_{i,n}^{(n)}\subset \si_{n,k}$ is the inverse image of the exceptional divisor $D_{i,n}^{(n)}\subset \si_{n,i+1}$ under the projection map $\si_{n,k}\to\si_{n,i+1}$.
\end{proposition}

\begin{proof}
(a) When $k=1$, we defined the divisor on $\si_{n,1}=F[n-1]\times S_1$ to be $D_{i,j}^{(n)}=D_{i,j}^{(n-1)}\times S_1$, so the conclusion holds. When $k>1$, recall that we have the following figure:
\[\begin{tikzcd}
\si_{n,k+1}=\Bl_{\tr_{k,n}}\si_{n,k}\arrow[r,"\pi_{n,k+1}"] & \si_{n,k}\arrow[r] & F[n-1]\arrow[l, bend left,"\tr_{k,n}"].
\end{tikzcd}\]

The divisor $D_{i,j}^{(n)}\subset \si_{n,k+1}$ is the strict transform of the divisor with the same name coming from $\si_{n,k}$. We claim that the strict transform coincides with the total transform. By construction, the moduli space $F[n-1]$ is irreducible and the map $\tr_{k,n}$ is a regular embedding, so the exceptional divisor $D_{k,n}^{(n)}$ of the blowup $\si_{n,k+1}\to \si_{n,k}$ is also irreducible. Hence, the only way statement (a) could fail is if the exceptional divisor $D_{k,n}^{(n)}\subset\si_{n,k+1}$ is completely contained in the inverse image $\pi_{n,k+1}^{-1}(D_{i,j}^{(n)})$, which could happen only if the blowup locus $\tr_{k,n}(F[n-1])\subset \si_{n,k}$ was completely contained in $D_{i,j}^{(n)}$. However, the projection $\si_{n,k}\to F[n-1]$ maps $\tr_{k,n}(F[n-1])\twoheadrightarrow F[n-1]$ and $D_{i,j}^{(n)}\twoheadrightarrow D_{i,j}^{(n-1)}$, so the induction step is complete and claim (a) is correct.

(b) When $k>i+1$, we have a similar picture as before:
\[\begin{tikzcd}
\si_{n,k+1}=\Bl_{\tr_{k,n}}\si_{n,k}\arrow[r,"\pi_{n,k+1}"] & \si_{n,k}\arrow[r] & F[n-1]\arrow[l, bend left,"\tr_{k,n}"].
\end{tikzcd}\]

The divisor $D_{i,n}^{(n)}\subset \si_{n,k+1}$ is defined to be the strict transform of the divisor with the same name coming from $\si_{n,k}$. Using the same argument as in (a), it is enough to prove that the blowup locus $\tr_{k,n}(F[n-1])$ is not fully contained in $D_{i,n}^{(n)}$. We prove this inductively over $k=i+1,\dots, n$. 

When $k=i+1$, a $B$-point of $D_{i,n}^{(n)}\subset \si_{n,i+1}$ is of the form $(p_1,\dots, p_{n-1}; p'_{i+1})$, where $p'_{i+1}\in \si_{i+1}$ maps to $p_i$ under the projection map $\si_{B,i+1}\to \si_{B,i}$, and a $B$-point of $\tr_{i+1,n}(F[n-1])\subset \si_{n,i+1}$ is of the form $(p_1,\dots,p_{n-1};p_{i+1})$. Clearly $\tr_{i+1,n}(F[n-1])$ is not fully contained in $D_{i,n}^{(n)}$.

In the induction step, we know that $D_{i,n}^{(n)}\subset \si_{n,k}$ is the inverse image of the exceptional divisor of the blowup $\si_{n,i+1}\to \si_{n,i}$, hence a $B$-point of $D_{i,n}^{(n)}\subset \si_{n,k}$ can be summarized as $(p_1,\dots, p_{n-1};p_{k}')$, where $p_{k}'\in \si_{B,k}$ maps to $p_{i}$ under the projection map $\si_{B,k}\to \si_{B,i}$. We conclude $\tr_{k,n}(F[n-1])\not\subset D_{i,n}^{(n)}$ and the proof is complete.

\end{proof}

\begin{proposition}\label{p-i-j}
Let $1\leq i<j\leq n$ and $D_{i,j}^{(n)}$ be a divisor of $F[n]$ as defined above. A $B$-point of $F[n]$, denoted by a tuple $(p_1, p_2,\dots, p_{n})$ consistent with Notation~\ref{k-points}, is a point of $D_{i,j}^{(n)}$ if and only if the projection $\si_{B,j}\to \si_{B,i}$ maps the point $p_{j}$ to the point $p_{i}$.
\end{proposition}

\begin{proof}
When $j<n$, $D_{i,j}^{(n)}$ is the inverse image of $D_{i,j}^{(n-1)}$ under the forgetful map $F[n]\to F[n-1]$. Hence we can assume without loss of generality that $j=n$. By Proposition~\ref{full}, the divisor $D_{i,n}^{(n)}\subset F[n]=\si_{n,n}$ is the inverse image of the blowup locus $\tr_{i,n}(F[n-1])\subset \si_{n,i}$ under the projection map $\si_{n,n}\to \si_{n,i}$:



\[\begin{tikzcd}
D_{i,n}^{(n)}=\pr^{-1}(\tr_{i,n}(F[n-1])\subset \si_{n,n}\arrow[r,"\pr"] & \si_{n,i}\arrow[r] & F[n-1]\arrow[l, bend left,"\tr_{i,n}"].
\end{tikzcd}\]

By Proposition~\ref{beh-maps}, the morphism $\si_{n,n}\to \si_{n,i}$ maps points $(p_1,\dots,p_{n-1},p_n)\mapsto (p_1,\dots, p_{n-1};\ov{p_n})$, where $\ov{p_n}\in \si_{B,i}$ is the projection of $p_n$ under the map $\si_{B,n}\to \si_{B,i}$. On the other hand, the section $\tr_{i,n}: F[n-1]\to \si_{n,i}$ maps points $(p_1,\dots,p_{n-1})\mapsto (p_1,\dots, p_{n-1};p_i)$. In conclusion, a $B$-point $F[n]$ is a point of $D_{i,n}^{(n)}$ if and only if $p_n\to p_i$ under the map $\si_{B,n}\to \si_{B,i}$.\qedhere

\end{proof}


\begin{proposition}\label{div-sm}
Let $1\leq i<j\leq n$. The divisor $D_{i,j}^{(n)}\subset F[n]$ is smooth over $\Spec R$.
\end{proposition}

\begin{proof}
We prove this statement inductively over $n$. When $n=2$, we have only one divisor of this form, namely $D_{1,2}^{(2)}$, which is the exceptional divisor of the blowup of $S_1\times S_1$ along the diagonal, so it is clearly smooth. Inductively, we need to analyze two cases: either $j<n$ or $j=n$. 

Assume $j<n$. We show inductively over $k$ that $D_{i,j}^{(n)}\subset \si_{n,k}$ is smooth. As the base case $k=1$, recall that the divisor $D_{i,j}^{(n)}\subset \si_{n,1}=F[n-1]\times S_1$ is defined to be $D_{i,j}^{(n-1)}\times S_1$, hence it is smooth by the induction hypothesis. Now assume $D_{i,j}^{(n)}\subset \si_{n,k}$ is smooth. We know that $\si_{n,k+1}=\Bl_{\tr_{k,n}}\si_{n,k}$ and $D_{i,j}^{(n)}\subset \si_{n,k+1}$ is the strict transform of the divisor with the same name coming from $\si_{n,k}$. We claim that smoothness is preserved because the intersection of this divisor with the blowup locus $D_{i,j}^{(n)}\cap\tr_{k,n}(F[n-1])$ inside $\si_{n,k}$ is itself smooth. In fact, we show it satisfies the following isomorphism:
\[D_{i,j}^{(n)}\cap \tr_{k,n}(F[n-1])\cong D_{i,j}^{(n-1)}.\]

As a result of Proposition~\ref{beh-maps} and Proposition~\ref{full}, a $B$-point of $D_{i,j}^{(n)}\subset \si_{n,k}$ has the form $(p_1,\dots,p_{n-1};p'_k)$, where $p_j\to p_i$ under the projection map $\si_{B,j}\to \si_{B,i}$, and a $B$-point of $\tr_{k,n}(F[n-1])\subset \si_{n,k}$ has the form $(p_1,\dots, p_{n-1};p_k)$. In conclusion, a $B$-point of $D_{i,j}^{(n)}\cap \tr_{k,n}(F[n-1])$ can be summarized as $(p_1,\dots, p_{n-1}; p_k)$, where $p_{j}\to p_{i}$ under the projection map $\si_{B,j}\to\si_{B,i}$. This means the functors of points of $D_{i,j}^{(n)}\cap \tr_{k,n}(F[n-1])$ and $D_{i,j}^{(n-1)}$ are isomorphic, hence the two schemes are isomorphic, proving our claim.

Now assume that $j=n$. We show inductively that $D_{i,n}^{(n)}\subset \si_{n,k}$ is smooth, for all $i+1\leq k\leq n$. When $k=i+1$, $D_{i,n}^{(n)}\subset \si_{n,i+1}$ is the exceptional divisor of the blowup of $\si_{n,i}$ along the locus $\tr_{i,n}(F[n-1])$. By Theorem~\ref{main2}, $\si_{n,i}$ is smooth and the blowup locus $\tr_{i,n}(F[n-1])$ is regularly embedded, so this exceptional divisor $D_{i,n}^{(n)}\subset \si_{n,i+1}$ is smooth. For the induction step we show, as before, that the intersection $D_{i,j}^{(n)}\cap\tr_{k,n}(F[n-1])\subset \si_{n,k}$ satisfies the following isomorphism, and hence it is smooth:
$$D_{i,n}^{(n)}\cap \tr_{k,n}(F[n-1])\cong D_{i,k}^{(n-1)}.$$

As a result of Proposition~\ref{beh-maps} and Proposition~\ref{full}, a $B$-point of $D_{i,n}^{(n)}\subset \si_{n,k}$ has the form $(p_1,\dots,p_{n-1};p'_k)$, where $p'_k\to p_i$ under the projection map $\si_{B,k}\to \si_{B,i}$, and a $B$-point of $\tr_{k,n}(F[n-1])\subset \si_{n,k}$ has the form $(p_1,\dots, p_{n-1};p_k)$. In conclusion, a $B$-point of $D_{i,n}^{(n)}\cap \tr_{k,n}(F[n-1])$ can be identified as $(p_1,\dots, p_{n-1}; p_{k})$, where  $p_{k}\to p_{i}$ under the projection map $\si_{B,k}\to\si_{B,i}$. This means the functors of points of $D_{i,j}^{(n)}\cap \tr_{k,n}(F[n-1])$ and $D_{i,k}^{(n-1)}$ are isomorphic, hence the two schemes are isomorphic.\qedhere
\end{proof}

\section{The Chow ring of the moduli scheme}
\label{chow-ring}

\hspace{12pt} In this section we assume that the underlying ring $R$ is an algebraically closed field $k$. In this new setup, the moduli space $F[n]$ and all the schemes in its universal family are smooth projective varieties. 

The main result here is that the Chow ring of the moduli space $F[n]$ is generated by the classes of the divisors $\{D_{i,j}^{(n)}\}_{1\leq i<j\leq n}$ over the Chow ring $\A^*(S_1^n)$. We conclude the section by proving certain key relations among these classes in the Chow ring $\A^*(F[n])$. In the next and final section, we prove that these relations are sufficient to give a precise description of the Chow ring $\A^*(F[n])$ in the special case where $S_1$ is a rational surface and the base field is the complex numbers $\mathbb{C}$.

We start with a theorem by Keel, which is the key ingredient in our proof:

\begin{theorem}\label{keel}
Let $Y$ be a variety and let $i:X\hookrightarrow Y$ be a regularly embedded subvariety. Let $\tilde{Y}$ be the blowup of $Y$ along $X$. Suppose the map of bivariate rings $i^*:\A^*(Y)\to \A^*(X)$ is surjective. Then:

\[\A^*(\tilde{Y})\cong \frac{\A^*(Y)[T]}{(P(T),T\cdot ker(i^*))},\]
where $P(T)\in\A^*(Y)[T]$ is any polynomial whose constant term is $[X]$ and whose restriction to $\A^*(X)$ is the Chern polynomial of the normal bundle $N=N_XY$, i.e.:
\[i^*P(T)=T^d+T^{d-1}c_1(N)+\dots+c_{d-1}(N)T+c_d(N),\]
where $d=codim(X,Y)$. This isomorphism is induced by
\[\pi^*:\A^*(Y)[T]\to\A^*(\tilde{Y})\]
and by sending $-T$ to the class of the exceptional divisor.
\end{theorem}

\begin{proof}
See~\cite{keel}, Appendix, Theorem 1.
\end{proof}

\begin{remark}
The moduli space $F[n]$ is a smooth projective variety, for any $n\geq 1$, hence its bivariate ring $\A^*(F[n])$ is isomorphic to its Chow ring $\mathbb{CH}^*(F[n])$ (see~\cite{fulton}, Ch. 17).
\end{remark}

\begin{corollary}\label{chow2} Let $1\leq i< n$. Let $\si_{n,i+1}$ and $\si_{n,i}$ be two of the varieties in the universal family over the moduli variety $F[n-1]$. The Chow ring of $\si_{n,i+1}$ has the following description:
\[\A^*(\si_{n,i+1})\cong \frac{\A^*(\si_{n,i})[D_{i,n}^{(n)}]}{\langle P_{i,n}(-D_{i,n}^{(n)}), D_{i,n}^{(n)}\cdot \ker(\tr_{i,n}^*)\rangle},\]
where $P_{i,n}$ is a quadratic polynomial with coefficients in $\A^*(\si_{n,i})$.
\end{corollary}

\begin{proof}
In the universal family over the moduli variety $F[n-1]$, $\si_{n,i+1}$ is the blowup of $\si_{n,i}$ along the locus $\tr_{i,n}(F[n-1])$, and the corresponding exceptional divisor is $D_{i,n}^{(n)}$. By Theorem~\ref{main2}, $\si_{n,i}$ is a variety and $\tr_{i,n}:F[n-1]\hookrightarrow \si_{n,i}$ is a regularly embedded subvariety. Additionally, $\tr_{i,n}$ is a section of the projection $\si_{n,i}\to F[n-1]$, hence the corresponding map on Chow rings $\tr_{i,n}^*:\A^*(\si_{n,i})\to \A^*(F[n-1])$ is surjective. The conclusion follows immediately as an application of Theorem~\ref{keel}.\qedhere
\end{proof}

\begin{remark}\label{rmkk}
By Proposition~\ref{full}, the induced map on Chow rings: $\pi_{n,i+1}^*:\A^*(\si_{n,i})\to\A^*(\si_{n,i+1})$ sends the class of any divisor $D_{j,k}^{(n)}$ to the class of the divisor with the same name inside $\A^*(\si_{n,i})$. This means that the notation remains consistent in all statements of this section.
\end{remark}



\begin{theorem}\label{chow}
With notation as in Corollary~\ref{chow2}, the Chow ring of the moduli space $F[n]$ is as follows:
\[\A^*(F[n])\cong \frac{\A^*(S_1^n)[D_{i,j}^{(n)}]_{1\leq i<j\leq n}}{\langle D_{i,j}^{(n)}\cdot ker(\tr_{i,j}^*), P_{i,j}(-D_{i,j}^{(n)})_{1\leq i<j\leq n}\rangle}.\]
\end{theorem}

\begin{proof}
To generalize Corollary~\ref{chow2}, we first need to give an alternative way of defining the divisors $D_{i,j}^{(n)}\subset F[n]$. Recall that the natural projection map $F[n]\to S_1^n$ decomposes as a series of blowups; more specifically, we encounter the following situation:
\[\begin{tikzcd}[column sep=tiny]
F[n]=\si_{n,n}\arrow[r] &\dots\arrow[r] &\si_{j,i+1}\times S_1^{n-j}\arrow[r] & \si_{j,i}\times S_1^{n-j}\arrow[r] &\dots\arrow[r] &F[j]\times S_1^{n-j}\to \dots\to S_1^n\arrow[ll, bend left,"\tr_{i,j}\times \id"].
\end{tikzcd}\]

The scheme $\si_{j,i+1}\times (S_1)^{n-j}$ is the blowup of $\si_{j,i}\times (S_1)^{n-j}$ along the locus $\tr_{i,j}\times \id$. We define $D_{i,j}^{(n)}\subset \si_{j,i+1}\times (S_1)^{n-j}$ to be the exceptional divisor of this blowup. By abuse of notation, we define $D_{i,j}^{(n)}\subset F[n]=\si_{n,n}$ to be the strict transform of this exceptional divisor in the tower of blowups. Following an identical argument as the one in Proposition~\ref{full}, more is true: $D_{i,j}^{(n)}\subset F[n]$ is actually the full inverse image of the exceptional divisor in the tower of blowups. Therefore, by Theorem~\ref{keel}, we conclude that: 
$$\A^*(\si_{j,i+1}\times S_1^{n-j})\cong \frac{\A^*(\si_{j,i}\times S_1^{n-j})[D_{i,j}^{(n)}]}{\langle P_{i,j}(-D_{i,j}^{(n)}), D_{i,j}^{(n)}\cdot \ker(\tr_{i,j}^*)\rangle},$$
where $P_{i,j}$ is a quadratic polynomial with coefficients in $\A^*(\si_{j,i}\times(S_1)^{n-j})$. 

Applying this procedure step by step from $S_1^n$ all the way up to $\si_{n,n}=F[n]$, we immediately obtain the desired formula for the Chow ring of the moduli variety $F[n]$.\qedhere
\end{proof}

\begin{proposition}\label{relations}
Let $d$ be a divisor class of $S_1$ and $d_i^*\in \A^*(F[n])$ the image of $d$ under the composed projection $F[n]\to S_1^n\xrightarrow{\pr_i}S_1$. The following relations hold in the Chow ring $\A^*(F[n])$:
\begin{enumerate}
\item[(i)] $D_{i,j}^{(n)}d_i^*=D_{i,j}^{(n)}d_j^*, \forall 1\leq i<j\leq n$;
\item[(ii)] $D_{i,j}^{(n)}D_{j,k}^{(n)}=D_{i,k}^{(n)}D_{j,k}^{(n)},\forall 1\leq i<j<k\leq n.$
\end{enumerate}
\end{proposition}

\begin{proof}

(i) This is true because the natural projection $F[n]\to S_1^n$ maps $D_{i,j}^{(n)}$ surjectively onto the diagonal $\tr_{ij}$, for all $1\leq i<j\leq n$.

(ii) Intuitively, the relation holds because the left-hand side parametrizes tuples $(p_1,\dots,p_n)$ in which $p_j\mapsto p_i$ and $p_k\mapsto p_j$, while the right-hand side parametrizes tuples in which $p_k\mapsto p_i$ and $p_k\mapsto p_j$.

We first reduce the proof to the case where $k=n$. If $k<n$, we can assume inductively that a similar relation holds in the Chow ring of $F[n-1]$:
\[D_{i,j}^{(n-1)}D_{j,k}^{(n-1)}=D_{i,k}^{(n-1)}D_{j,k}^{(n-1)}.\]

Now recall that inside the variety $\si_{n,1}=F[n-1]\times S_1$, the divisors $D_{i,j}^{(n)}, D_{i,k}^{(n)}, D_{j,k}^{(n)}$ are defined to be $D_{i,j}^{(n-1)}\times S_1,D_{i,k}^{(n-1)}\times S_1,D_{j,k}^{(n-1)}\times S_1$, respectively, so relation (ii) holds inside the Chow ring $\A^*(\si_{n,1})$. By Remark~\ref{rmkk}, the same relation lifts to the Chow ring $\A^*(F[S_1,n])$.

We are left to show that (ii) holds when $k=n$. By Remark~\ref{rmkk}, it suffices to show it holds inside the Chow ring of $\si_{n,j+1}$. The Chow ring of $\si_{n,j+1}$ is given as follows:
\[\A^*(\si_{n,j+1})\cong \frac{\A^*(\si_{n,j})[D_{j,n}^{(n)}]}{\langle P_{j,n}(-D_{j,n}^{(n)}), D_{j,n}^{(n)}\cdot \ker(\tr_{j,n}^*)\rangle},\]
hence we need to show that: 
\[D_{i,j}^{(n)}-D_{i,n}^{(n)}\in \ker(\tr_{j,n}^*:\A^*(\si_{n,j})\to\A^*(F[n-1])).\]


Consider first the following diagram:

\[
\begin{tikzcd}
D_{i,n}^{(n)}&[-28pt]\subset\si_{n,i+1}\arrow[d]\arrow[dr, phantom, "\ulcorner", very near start]\arrow[r]&\si_{j,i+1}\arrow[d] &[-36pt]\supset D_{i,j}^{(j)}\\
&\si_{n,i}\arrow[d]\arrow[dr, phantom, "\ulcorner", very near start]\arrow[r]&\si_{j,i}\arrow[d]&\\
&F[n-1]\arrow[r]\arrow[u, "\tr_{i,n}", shift left=3ex, bend left]&F[j-1].\arrow[u, "\tr_{i,j}"', shift right=3ex, bend right]&
\end{tikzcd}
\]

The bottom square is cartesian and the section $\tr_{i,n}$ is the pullback of $\tr_{i,j}$ along the projection map $F[n-1]\to F[j-1]$, hence the exceptional divisor $D_{i,n}^{(n)}\subset \si_{n,i+1}$ is the pullback of the exceptional divisor $D_{i,j}^{(j)}\subset \si_{j,i+1}$. In other words, $D_{i,n}^{(n)}\cong F[n-1]\times_{F[j-1]}D_{i,j}^{(j)}$ inside $\si_{n,i+1}$. This relation lifts to $\si_{n,j}$. On the other hand, we have that $\si_{n,j}\cong F[n-1]\times_{F[j-1]}F[j]$. By Proposition~\ref{full}, under this isomorphism, $D_{i,j}^{(n)}\cong D_{i,j}^{(n-1)}\times_{F[j-1]}F[j]$. 





In conclusion, inside $\si_{n,j}\cong F[n-1]\times_{F[j-1]}F[j]$, we have $D_{i,j}^{(n)}\cong D_{i,j}^{(n-1)}\times_{F[j-1]}F[j]$ and $D_{i,n}^{(n)}\cong F[n-1]\times_{F[j-1]}D_{i,j}^{(j)}$. Additionally, the morphism $\tr_{j,n}:F[n-1]\to \si_{n,j}=F[n-1]\times_{F[j-1]}F[j]$ acts like a `truncated` diagonal embedding, so the proof is complete:
\[D_{i,j}^{(n)}-D_{i,n}^{(n)}\in \ker(\tr_{j,n}^*:\A^*(\si_{n,j})\to\A^*(F[n-1])).\qedhere\]
\end{proof}

\section{The Chow ring of the moduli scheme for rational surfaces}
\label{chow-ring-rational}

As a special case of the theory developed above, we give a precise description of the Chow ring $\A^*(F[n])$ when the base surface $S_1$ is a smooth projective rational surface over the complex numbers ($\Spec R=\Spec \mathbb{C}$). The result relies on a few key ideas. First, the canonical map $\text{cl}:\A^*(F[n])\to\mathbb{H}^{2*}(F[n])$ is an isomorphism. Second, for some prime $p$ and any $q=p^l$, where $l\gg 0$, we can define the moduli space $F[n]\otimes \mathbb{F}_q$ over the finite field $\mathbb{F}_q$. The number of $\mathbb{F}_q$-points on $F[n]\otimes \mathbb{F}_q$ is given by a polynomial $R_n(q)$ that coincides with the Poincare polynomial of $F[n]$. We use this fact to derive precise formulas for the Betti numbers of $F[n]$; using these formulas, we show that the relations in Proposition~\ref{relations} are enough to give a complete description of $\A^*(F[n])$.

\begin{definition}
A scheme $X$ of characteristic zero is called an HI (for Homology Isomorphism) scheme if the canonical map from the Chow groups of $X$ to the homology groups is an isomorphism:
\[\A_*(X)\xrightarrow[\text{cl}]{\cong}\mathbb{H}_{2*}(X).\]
\end{definition}

\begin{proposition}\label{HI}
Let $Y$ be a variety and $i:X\hookrightarrow Y$ a regularly embedded subvariety. Let $\tilde{Y}$ be the blowup of $Y$ along $X$. If $X$ and $Y$ are HI schemes, then so is $\tilde{Y}$.
\end{proposition}

\begin{proof}
See~\cite{keel}, Appendix, Theorem 2.
\end{proof}

\begin{proposition}\label{fn-hi}
Let $S_1$ be a smooth projective variety over an algebraically closed field $k$ of characteristic zero. If $S_1^n$ is an HI scheme, for all $n\geq 0$, then so is $F[n]$.
\end{proposition}

\begin{proof}
We prove inductively over $n$ something stronger: $\forall n,j\geq 0$, the variety $F[n]\times S_1^j$ is an HI scheme. The base cases follow immediately, since $F[0]\cong \Spec k$ and $F[1]\cong S_1$.

For the induction step, assume that $F[n]\times S_1^j$ is an HI scheme, for all $j\geq 0$. Recall that we can obtain $F[n+1]$ from $F[n]\times S_1$ as a series of blowups:

\[\begin{tikzcd}[column sep=small]
F[n+1]=\si_{n+1,n+1}\arrow[r] &\si_{n+1,n}\arrow[r] &\dots\arrow[r] &\si_{n+1,2}\arrow[r] &\si_{n+1,1}=F[n]\times S_1\arrow[d]\\
 &&&&F[n].\arrow[u, "\tr_{1,n+1}", shift left=1ex, bend left]\arrow[ul, "\tr_{2,n+1}", bend left]\arrow[ulll, "\tr_{n,n+1}", bend left]
\end{tikzcd}\]

By the induction hypothesis, $F[n]$ and $F[n]\times S_1$ are HI varieties. By Theorem~\ref{main2}, the blowup locus $\tr_{i,n+1}:F[n]\hookrightarrow \si_{n+1,i}$ is a regular embedding, for all $1\leq i\leq n$. Applying Proposition~\ref{HI} repeatedly, we conclude step by step that $\si_{n+1,2},\si_{n+1,3},\dots, \si_{n+1,n+1}\cong F[n+1]$ are all HI varieties. More generally, we conclude that $F[n+1]\times S_1^j$ is an HI variety by considering a similar figure to the one above.\qedhere

\end{proof}

\begin{proposition}\label{rational-hi}
Let $k \geq 1$ and $S_1,\dots, S_k$ be complex smooth projective rational surfaces. Then $\prod_{i=1}^k S_i$ is an HI variety.
\end{proposition}

\begin{proof}
We prove this statement inductively over $k$. Let $S$ be a complex smooth projective rational surface. By the Enriques-Kodaira classification of complex surfaces (see\cite{barth-peters}, Part VI), there exist smooth projective surfaces $S_{n-1},\dots, S_1,S_0$, and morphisms $S=S_n\to S_{n-1}\to\dots\to S_1\to S_0$, such that each $S_{i+i}\to S_i$ is the contractions of a ($-1$)-curve and $S_0$ is a minimal rational surface (either $\mathbb{P}^2$ or the Hirzebruch surface $\mathbb{F}_a$, for $a=0$ or $a\geq 2$). Both $\mathbb{P}^2$ and $\mathbb{F}_a$ have algebraic cell decompositions (since they are toric), which means they are HI varieties. We apply Proposition~\ref{HI} repeatedly, obtaining step by step that $S_1,S_2,\dots, S_n=S$ are HI varieties, since each of them is obtained by blowing up a smooth HI surface at a smooth point.

Inductively, let $S_1,\dots,S_k$ be complex smooth projective rational surfaces. As before, each surface $S_i$ is obtained by blowing up a minimal rational surface $S_{i,0}$. The product $\prod_{i=1}^k S_{i}$ is thus obtained through series of blowups of the base product $\prod_{i=1}^k S_{i,0}$, which is an HI variety because it admits a cell decomposition. Every blowup locus is, by the induction hypothesis, an HI variety. We apply again Proposition~\ref{HI} repeatedly and conclude that every variety in the sequence of blowups is HI, finishing the proof.\qedhere
\end{proof}


\begin{corollary}
Let $S_1$ be a complex smooth projective rational surface and $F[S_1,n]=F[n]$ its associated moduli variety. There exists a canonical isomorphism:
\[\A^*(F[n])\xrightarrow{\cong}\mathbb{H}^{2*}(F[n]).\]
\end{corollary}

\begin{proof}
This is an immediate result of Proposition~\ref{fn-hi} and Proposition~\ref{rational-hi}.
\end{proof}

\begin{setup}\label{setup}
Let $S$ be a complex surface. There exists $R\subset \C$ a finitely generated $\Z$-algebra such that $S$ is defined over $R$, i.e. there exists a surface $S_R$ over $\Spec R$ such that the following square is cartesian:
\[
\begin{tikzcd}
S\arrow[d]\arrow[dr, phantom, "\ulcorner", near start]\arrow[r]&S_R\arrow[d]\\
\Spec \C\arrow[r]&\Spec R.
\end{tikzcd}
\]

For any $m\subset R$ maximal ideal, the field $\kappa(m)=R/m$ is finite, so there exists some prime $p$ and $q=p^l$, where $l\gg 0$, such that $\kappa(m)\subseteq \mathbb{F}_q$. We obtain the following figure:
\[
\begin{tikzcd}
S\arrow[d]\arrow[dr, phantom, "\ulcorner", near start]\arrow[r]&S_R\arrow[d]&S_{\kappa(m)}\arrow[dl, phantom, "\urcorner", very near start]\arrow[l]\arrow[d]&S_{\mathbb{F}_q}\arrow[dl, phantom, "\urcorner", near start]\arrow[l]\arrow[d]&S_{\mathbb{F}_{\overline{q}}}\arrow[dl, phantom, "\urcorner", near start]\arrow[l]\arrow[d]\\
\Spec \C\arrow[r]&\Spec R&\Spec \kappa(m)\arrow[l]&\Spec \mathbb{F}_q\arrow[l]&\Spec \mathbb{F}_{\overline{q}}\arrow[l].
\end{tikzcd}
\]

For the rest of this section, when we say a complex surface $S$ can be defined as a surface $S_{\mathbb{F}_q}$ over a finite field $\mathbb{F}_q$, it means we do a procedure as above.

\end{setup}

\begin{proposition}\label{r(t)}
Let $S$ be a complex smooth projective rational surface. There exists a prime integer $p$ and $q=p^l$, for some $l\gg 0$, such that $S$ can be defined as a smooth projective rational surface $S_{\mathbb{F}_q}$ over $\Spec \mathbb{F}_q$. For this choice of $p$ and $q$, there exists a quadratic polynomial $r(t)$ with the property that, for any $a\geq 1$ and $q'=q^a$, the number of $\mathbb{F}_{q'}$-points on $S_{\mathbb{F}_{q'}}$ equals $r(q')$.
\end{proposition}

\begin{proof}
Let $S$ be a complex smooth projective rational surface. By the Enriques-Kodaira classification of surfaces, there exist complex smooth projective surfaces $S_{n-1},\dots, S_1,S_0$, and a sequence of morphisms $S=S_n\to S_{n-1}\to\dots\to S_1\to S_0\to\Spec \C$, such that each $S_{i+i}\to S_i$ is the contractions of a ($-1$)-curve and $S_0$ is a minimal rational surface. As in the Setup~\ref{setup} above, we can find $R\subset \C$ a finitely generated $\Z$-algebra such that all the surfaces $S_i$ are defined over $\Spec R$. Moreover, we can pick $R$ in such a way that, if $S_0$ is $\P^2$ or $\mathbb{F}_a$, then the surface $S_{0,R}$ is either $\mathbb{P}^2_{R}$ or $\mathbb{F}_{a,R}$, respectively:
\[
\begin{tikzcd}[row sep=small, column sep=small]
S_n\arrow[d]\arrow[dr, phantom, "\ulcorner", very near start]\arrow[r]&S_{n,R}\arrow[d]&S_{n,\kappa(m)}\arrow[dl, phantom, "\urcorner", very near start]\arrow[l]\arrow[d]&S_{n,\mathbb{F}_q}\arrow[dl, phantom, "\urcorner", very near start]\arrow[l]\arrow[d]&S_{n,\mathbb{F}_{q'}}\arrow[dl, phantom, "\urcorner", very near start]\arrow[l]\arrow[d]\\
\vdots\arrow[d]\arrow[r]&\vdots\arrow[d]&\vdots\arrow[l]\arrow[d]&\vdots\arrow[l]\arrow[d]&\vdots\arrow[l]\arrow[d]\\
S_1\arrow[d]\arrow[dr, phantom, "\ulcorner", very near start]\arrow[r]&S_{1,R}\arrow[d]&S_{1,\kappa(m)}\arrow[dl, phantom, "\urcorner", very near start]\arrow[l]\arrow[d]&S_{1,\mathbb{F}_q}\arrow[dl, phantom, "\urcorner", very near start]\arrow[l]\arrow[d]&S_{1,\mathbb{F}_{q'}}\arrow[dl, phantom, "\urcorner", very near start]\arrow[l]\arrow[d]\\
S_0\arrow[d]\arrow[dr, phantom, "\ulcorner", very near start]\arrow[r]&S_{0,R}\arrow[d]&S_{0,\kappa(m)}\arrow[dl, phantom, "\urcorner", very near start]\arrow[l]\arrow[d]&S_{0,\mathbb{F}_q}\arrow[dl, phantom, "\urcorner", very near start]\arrow[l]\arrow[d]&S_{0,\mathbb{F}_{q'}}\arrow[dl, phantom, "\urcorner", very near start]\arrow[l]\arrow[d]\\
\Spec \C\arrow[r]&\Spec R&\Spec \kappa(m)\arrow[l]&\Spec \mathbb{F}_q\arrow[l]&\Spec \mathbb{F}_{q'}\arrow[l].
\end{tikzcd}
\]


In every blowup $S_{i+1}\to S_i$, we replace one smooth point of $S_i$ with a copy of $\mathbb{P}^1$, so the number of $\mathbb{F}_{q'}$-points on $S_{\mathbb{F}_{q'}}$ is given by a polynomial $r(q')$ that satisfies:
\[r(q')=\begin{cases}(q')^2+(n+1)q'+1,\text{ if }S_0=\mathbb{P}^2,\\
(q')^2+(n+2)q'+1,\text{ if }S_0=\mathbb{F}_a.\end{cases}\]\qedhere
\end{proof}

\begin{proposition}\label{formula}
Let $S_1$ be a complex smooth projective rational surface. Let $p$, $q$, and $r(t)$ be defined as in Proposition~\ref{r(t)}.For any $a\geq 1$ and $q'=q^a$, the number of $\mathbb{F}_{q'}$-points on the moduli space $F[n]_{\mathbb{F}_{q'}}=F[S_{1,\mathbb{F}_{q'}},n]$ is given by the following polynomial $R_n(q)$:
\[R_n(q')=\prod_{i=0}^{n-1}(r(q')+iq').\]
\end{proposition}

\begin{proof}
We prove the statement inductively, using the fact that the number of $\mathbb{F}_{q'}$-points of $S_{1,\mathbb{F}_{q'}}$ blown up at $n$ points equals $r(q')+nq'$. When $n=1$, $F[1]_{\mathbb{F}_{q'}}\cong S_{1,\mathbb{F}_{q'}}$ has exactly $r(q')$ points over $\mathbb{F}_{q'}$. For the induction step we need to show that:
\[R_{n+1}(q')=R_n(q')(r(q')+nq').\]

To see this, recall that we have a forgetful map $F[n+1]_{\mathbb{F}_{q'}}\to F[n]_{\mathbb{F}_{q'}}$. Under this map, the fiber of any point $x=(p_1,\dots, p_{n})\in F[n]_{\mathbb{F}_{q'}}$ is isomorphic to the blown up surface $S_{n+1,\mathbb{F}_{q'}}$, so it has $r(q')+nq'$ rational points.\qedhere
\end{proof}

\begin{definition}
Let $X$ be a smooth, irreducible complex algebraic variety. The Poincar\'{e} polynomial of $X$ is:
\[P_X(q)=\sum_{i=1}^{2\dim X}b_iq^i,\]
where $b_i$ is the rank of the $i^{th}$ singular homology group $H^i(X,\Z)$.
\end{definition}

\begin{lemma}\label{betti}
Let $S_1$ be a complex smooth projective rational surface. Let $p$ be a prime integer and $q=p^l$, as in Proposition~\ref{r(t)}. The Poincar\'{e} polynomial of $F[n]=F[S_1,n]$, denoted by $P_n(q)$, coincides with the polynomial $R_n(q)$ which gives the number of $\mathbb{F}_q$-points on the moduli space $F[n]_{\mathbb{F}_q}$.
\end{lemma}

\begin{proof}
Let $X=F[n]=F[S_1,n]$ be the moduli variety corresponding to $S_1$ over $\Spec \C$. We can regard $S_1$ as a rational surface $S_{1,\mathbb{F}_q}$ over $\mathbb{F}_q$, as in Proposition~\ref{r(t)} above. Let $X_{\mathbb{F}_q}$ be the moduli space associated to $S_{1,\mathbb{F}_q}$. Since $X$ is smooth and projective, the Betti numbers corresponding to the $l$-adic cohomology (where $l\neq 0 \mod p$) are independent of $l$, and coincide with the Betti numbers corresponding to the ordinary (integral) cohomology of the topological space $X$ (see~\cite{milne}):
\[b_i = \text{rk } H^i(X,\Z)=\text{rk } H^i(X,\Q)= \text{rk } H^i_{\acute{e}t}(X_{\ov{\mathbb{F}}_p},\mathbb{Q}_l).\]

One the other hand, recall the Grothendieck-Lefschetz Trace Formula (see~\cite{milne}, Thm. 13.4, p. 292), which states the following:
\[\# X(\mathbb{F}_q)=\sum_{i=0}^{2n}(-1)^i\tra(\Frob_q|H_c^i(X_{\ov{\mathbb{F}}_q},\mathbb{Q}_l)).\]

Since $X$ is proper, $H_c^i(X_{\ov{\mathbb{F}}_q},\mathbb{Q}_l)=H^i(X_{\ov{\mathbb{F}}_q},\mathbb{Q}_l)$. Moreover, as a consequence of the Weil conjectures (see~\cite{deligne-weil}), we have:
\[\tra(\Frob_q|H^i(X_{\ov{\mathbb{F}}_q},\mathbb{Q}_l))=z_{i,1}+\dots+z_{i,b_i},\]
where $z_{i,1},\dots,z_{i,b_i}$ are the eigenvalues of the Frobenius map, and they satisfy $|z_{i,j}|=q^{i/2}$, for all $j$. Now, if we replace the field $\mathbb{F}_q$ by $\mathbb{F}_{q'}$, where $q'=q^a$, then:
\[\tra(\Frob_{q^a}|H^i(X_{\ov{\mathbb{F}}_{q^a}},\mathbb{Q}_l))=z_{i,1}^a+\dots+z_{i,b_i}^a.\]

In conclusion, for all $a\geq 1$, $R_n(q^a)=\sum_{i=1}^{2n}(-1)^i\sum_{j=0}^{b_j}z_{i,j}^{a}$, where $|z_{i,j}|=q^{i/2}, \forall i,j$. It follows immediately that $H^{2i+1}(X_{\ov{\mathbb{F}}_{q^a}},\mathbb{Q}_l)=0$, for all $i$ and $l\neq 0 \mod p$, and $z_{2i,j}=q^{i}$, for all $i, j$. With this, we conclude our statement:
\[R_n(q)=P_n(q)=\sum_{i=1}^{n} b_{2i}q^i.\]\qedhere
\end{proof}

Let $r(q)=q^2+kq+1$ be the Poincar\'{e} polynomial of $S_1$. This means that $\A^*(S_1,\mathbb{Z})$ is generated in degree 1 by $k$ classes $d_1,\dots, d_k$, and by one class in degree 2.

\begin{theorem}
Let $S_1$ be a complex smooth projective rational surface. Let $\Pi:F[n]\to S_1^n$ be the natural projection map and $\pr_i:S_1^n\to S_1$ the projection onto the $i$-th copy, $\forall 1\leq i\leq n$. Let $\pr_i^*\circ \Pi^*:\A^*(S_1)\to \A(F[n])$ be the induced map on Chow rings and $d_{i,1},\dots, d_{i,k}$ the images of the classes $d_1,\dots, d_k$, respectively. The Chow ring of the moduli space $\A^*(F[n])$ is:
\[\A^*(F[n])\cong \frac{(\A^*(S_1))^{\otimes n}[D_{i,j}^{(n)}]_{1\leq i<j\leq n}}{\langle D_{j,k}^{(n)}(D_{i,j}^{(n)}-D_{i,k}^{(n)}), D_{j,k}^{(n)}(d_{i,j}-d_{i,k}), P_{i,j}(-D_{i,j}^{(n)})\rangle}.\]
\end{theorem}

\begin{proof}
Let $R=\frac{(\A^*(S_1))^{\otimes n}[D_{i,j}^{(n)}]_{1\leq i<j\leq n}}{\langle D_{j,k}^{(n)}(D_{i,j}^{(n)}-D_{i,k}^{(n)}), D_{j,k}^{(n)}(d_{i,j}-d_{i,k}), P_{i,j}(-D_{i,j}^{(n)})\rangle}$. We claim the following composition of morphisms is an isomorphism, after tensoring by $\Q$:
\[
R\twoheadrightarrow \A^*(F[n],\Z)\xrightarrow{\cong} \mathbb{H}^*(F[n],\Z)\twoheadrightarrow \mathbb{H}^*(F[n],\Q).
\]

By Theorem~\ref{chow}, we know that $\A^*(F[n])$ is generated over $\A^*(S_1^n)$ by the classes of the divisors $\{D_{i,j}^{(n)}\}_{1\leq i<j\leq n}$. Moreover, by Proposition~\ref{relations}, we know that the following relations hold in the Chow ring $\A^*(F[n])$:
\begin{align}\label{rels}
D_{j,k}^{(n)}(d_{i,j}-d_{i,k})&=0\nonumber\\
 D_{j,k}^{(n)}(D_{i,j}^{(n)}-D_{i,k}^{(n)})&=0\\
 P_{i,j}(-D_{i,j}^{(n)})&=0.\nonumber
\end{align}

We show that the relations above are sufficient by looking at the Betti numbers of the moduli space $F[n]$. By definition, the $j$-th Betti number of $F[n]$, denoted by $b_{n,j}$, represents the number of codimension $j$ linearly independent generators of $\A^*(F[n])$ as a $\mathbb{Z}$-module. We obtain the following recursive relation from Proposition~\ref{formula}:
\[b_{n+1,j}=b_{n,j}+(n+k)b_{n,j-1}+b_{n,j-2}.\]

We give the following interpretation to the relation above: consider the map on Chow rings $\Pi_{n+1}^*:\A^*(F[n])\to \A^*(F[n+1])$ corresponding to the forgetful functor. Compared to the moduli space $F[n]$, the space $F[n+1]$ has $n+k$ extra divisors: $d_{n+1,1},\dots, d_{n+1,k}, D_{1,n+1}^{(n+1)},\dots D_{n,n+1}^{(n+1)}$. A generator in $\A^j(F[n+1])$ is either a class inherited from $\A^j(F[n])$ under the map $\Pi_{n+1}^*$ (this accounts for $b_{n,j}$ generators), or it is a product between a generator class coming from $\A^{j-1}(F[n])$ and one of the $n+k$ new divisor classes (this accounts for $(n+k)b_{n,j-1}$ generators), or it is a product between a generator class coming from $\A^{j-2}(F[n])$ and the one generator class coming from $\A^2(S_1)$ under the projection map $\pr_{n+1}:S_1^{n+1}\to S_1$ (this accounts for $b_{n,j-2}$ generators). It is easy to see that these are the only generators, since the divisors $D_{i,j}^{(n)}$ satisfy the identities of Equation~\ref{rels}.
\end{proof}

\begin{corollary}\label{pp2}
When $S_1=\mathbb{P}^2_{\C}$, the Chow ring of the moduli space $\A^*(F[\mathbb{P}^2,n])$ is:
\[\A^*(F[\mathbb{P}^2,n])\cong \frac{\mathbb{Z}[H_i^*, D_{j,k}^{(n)}]_{1\leq i\leq n, 1\leq j<k\leq n}}{\langle D_{j,k}^{(n)}(D_{i,j}^{(n)}-D_{i,k}^{(n)}), D_{j,k}^{(n)}(H_j^*-H_k^*), H_i^{*3}, P_{i,j}(-D_{i,j}^{(n)})\rangle},\]
where, $\forall 1\leq i\leq n$, $H_i^*$ is the image of the hyperplane class $H\in \A^*(\mathbb{P}^2)$ under the composition $F[\mathbb{P}^2,n]\to (\mathbb{P}^2)^n\xrightarrow{\pr_i} \mathbb{P}^2$. 
\end{corollary}

\end{document}